\documentclass[oneside,english]{amsart}
\usepackage{lmodern}
\usepackage[T1]{fontenc}
\usepackage[latin9]{inputenc}
\usepackage{geometry}
\usepackage{mathrsfs}
\usepackage{amsthm}
\usepackage{amssymb}

\makeatletter
\numberwithin{equation}{section}
\numberwithin{figure}{section}
\theoremstyle{plain}
\newtheorem{thm}{\protect\theoremname}
\theoremstyle{remark}
\newtheorem{rem}[thm]{\protect\remarkname}
\theoremstyle{plain}
\newtheorem{lem}[thm]{\protect\lemmaname}
\theoremstyle{plain}
\newtheorem{cor}[thm]{\protect\corollaryname}

\makeatother

\usepackage{babel}
\providecommand{\corollaryname}{Corollary}
\providecommand{\lemmaname}{Lemma}
\providecommand{\remarkname}{Remark}
\providecommand{\theoremname}{Theorem}

\begin{document}
\global\long\def\e{e}%
\global\long\def\V{{\rm Vol}}%
\global\long\def\bs{\boldsymbol{\sigma}}%
\global\long\def\bx{\mathbf{x}}%
\global\long\def\bz{\mathbf{z}}%
\global\long\def\by{\mathbf{y}}%
\global\long\def\bv{\mathbf{v}}%
\global\long\def\bu{\mathbf{u}}%
\global\long\def\bn{\mathbf{n}}%
\global\long\def\bM{\mathbf{M}}%
\global\long\def\bG{\mathbf{G}}%
\global\long\def\diag{{\rm diag}}%
\global\long\def\sol{\mathcal{Z}}%
\global\long\def\SN{\mathbb{S}^{N}}%
\global\long\def\SNK{\mathbb{S}^{N-\K}}%
\global\long\def\SNt{\mathbb{S}^{N-1}}%
\global\long\def\E{\mathbb{E}}%
\global\long\def\P{\mathbb{P}}%
\global\long\def\R{\mathbb{R}}%
\global\long\def\NN{\mathbb{N}}%
\global\long\def\BN{\mathbb{B}^{N}}%
\global\long\def\N{\mathsf{N}}%
\global\long\def\indic{\mathbf{1}}%
\global\long\def\M{K}%
\global\long\def\K{K}%
\global\long\def\J{J}%
\global\long\def\HaussN{\mathcal{H}_{N}}%
\newcommand{\WLOG}{{\sffamily{WLOG}}}
\title{Concentration for the zero set of large random polynomial systems }
\author{Eliran Subag}
\begin{abstract}
For random systems of $\K$ polynomials in $N+1$ real variables which
include those of Kostlan (1987) and Shub and Smale (1993), we
prove that the number of zeros on the unit sphere for $\K=N$ or the Hausdorff measure of the zero
set for $\K<N$ concentrates around its mean as
$N\to\infty$. To prove concentration we show that the variance of
the latter random variable normalized by its mean goes to zero. The
polynomial systems we consider depend on a set of parameters
which determine the variance of their Gaussian coefficients. We prove
that the convergence is uniform in those parameters and $\K$.
\end{abstract}

\maketitle

\section{Introduction }

Kostlan \cite{Kostlan1987,Kostlan1993,Kostlan2000} and Shub and Smale \cite{ShubSmaleI,ShubSmaleII,ShubSmaleIII,ShubSmaleIV,ShubSmaleV} studied large systems of random homogeneous polynomials in many variables.
Their original motivation, and the subject of Smale's well-known 17th problem \cite{SmaleProblems} was algorithmic questions about root-finding. In the complex case, after extensive study the 17th problem was resolved in recent years \cite{BeltranPardo2,BeltranPardo1,BurgisserCucker,Lairez1,Lairez2}. In the real, more difficult case, however, much less is known. 
Our goal in the current work is to understand the size of the set of solutions. While its expectation is known since the works of Kostlan, Shub and Smale from the 90s, its typical behavior is known only when the polynomials are identically distributed and the number of equations is maximal \cite{Wschebor}.
For systems of real random polynomials which include as a specific case those of Kostlan, Shub and Smale we prove, by a second moment argument, that the size of the set of solutions is typically close to its expectation, as the number of variables tends to infinity.

Let $N\in\mathbb{N}$ and $\mathbf{G}_{p}=(g_{i_{1},\dots,i_{p}}: 0\leq i_{1},\dots,i_{p}\leq N)$ be random tensors
whose elements are i.i.d. $\N(0,1)$ random variables, independent also for different $p$. 
Given some $d\geq 2$ and deterministic $a_{p}\in \R$,
\begin{equation}
	f_{N}(\bx):=\sum_{p=2}^{d}a_{p}\langle\mathbf{G}_{p},\bx^{\otimes p}\rangle:=\sum_{p=2}^{d}a_{p}\sum_{i_{1},\ldots,i_{p}=0}^{N}g_{i_{1},\dots,i_{p}}x_{i_{1}}\cdots x_{i_{p}}\label{eq:random_poly}
\end{equation}
is a random polynomial function of degree $d$ in $N+1$ real variables $\bx=(x_{0},x_{1},\ldots,x_{N})\in\R^{N+1}$.
 One can check that the centered Gaussian process $f_{N}(\bx)$ has the covariance function
$\E f_{N}(\bx)f_{N}(\by)=\xi(\bx\cdot\by)$ for $\xi(t):=\sum_{p=2}^{d}a_{p}^{2}t^{p}$.
The law of $f_{N}(\bx)$ is therefore invariant
under rotations.

We study random polynomial systems  of the form $\vec{f}_{N}(\bx):=(f_{N}^{(1)}(\bx),\ldots,f_{N}^{(N)}(\bx))$.
We will assume that the polynomials $f_{N}^{(k)}(\bx)$ are independent
of each other and that each is equal in distribution to $f_{N}(\bx)$ as above for some deterministic
$d=d_{N,k}$ and $a_p=a_{p,N,k}$. 
We will encode these parameters through $\Xi:=(\xi_{N,k})_{N=1,k=1}^{\infty,N}$ where $\xi_{N,k}(t):=\sum_{p=2}^{d_{N,k}}a_{p,N,k}^{2}t^{p}$.

Given some $\K\leq N$, consider the zero set of the first $\K$ polynomials
on the unit sphere $\SN\subset\R^{N+1}$,
\begin{equation}
\left\{ \bx:\,f_{N}^{(1)}(\bx)=\cdots=f_{N}^{(\K)}(\bx)=0,\,\|\bx\|=1\right\} .\label{eq:zeroset}
\end{equation}
This set is almost surely finite when $\K=N$ or a manifold of dimension
$N-\K$ when $\K<N$. We define $\sol_{N,\K}=\sol_{N,\K}(\Xi)$ as
the cardinality of the set in the former case, and as the $N-\K$
dimensional Hausdorff measure in the latter. A standard (by now) application
of the Kac-Rice formula \cite{Kac,Rice} allows one
to compute the \emph{expectation} of $\sol_{N,\K}$.
\begin{thm}
[Expected value]\label{thm:1stmoment}For any $N$, $\K$ and $\Xi$,
\begin{equation}
\E\sol_{N,\K}=\V(\SNK)\prod_{k=1}^{\M}\sqrt{\frac{\xi_{N,k}'(1)}{\xi_{N,k}(1)}},\label{eq:1stmoment}
\end{equation}
where $\V(\mathbb{S}^{d-1})=2\frac{\pi^{\frac{d}{2}}}{\Gamma(\frac{d}{2})}$
is the volume of the unit sphere in $\R^{d}$.
\end{thm}

Our main result is that the \emph{variance} of $\sol_{N,\K}$ normalized
by its mean tends to zero as $N\to\infty$ and this random variable thus concentrates,
\emph{uniformly} in $\Xi$ and $\K$. If $\xi_{N,k}(t)=t^{d_{N,k}}$ for all $N,k$, then $\vec{f}_{N}(\bx)$
is a system of homogeneous polynomials. Denote by $\mathscr{H}$ the
class of all such $\Xi$ and denote by $\mathfrak{\mathscr{H}}_{d}$
its subclass of constant degree with $d_{N,k}=d$ for all $N$ and $k$. The former
are sometimes called mixed homogeneous systems and the latter non-mixed homogeneous systems. 

\begin{thm}
[Concentration]\label{thm:2ndmoment-O(1)} Fix $\alpha \in(0,\frac12\log 2)$. Then, as $N\to\infty$,
\begin{equation}
\sup_{\K\leq N}\sup_{\Xi}{\rm Var}\left(\frac{\sol_{N,\K}}{\E\sol_{N,\K}}\right)\leq  N^{-\frac12}(2+ o(1)),\label{eq:momentsmatching}
\end{equation}
where the supremum is either over all $\Xi$ such that $\max_{k\leq N}\deg(\xi_{N,k})\leq e^{e^{N\alpha}}$ or over $\mathscr{H}$.
\end{thm}
Of course, by Chebyshev's inequality,  
\begin{equation}
\sup_{\K\leq N}\sup_{\Xi}\P\bigg(\,\bigg|\frac{\sol_{N,\K}}{\E\sol_{N,\K}}-1\bigg|>\epsilon\bigg)\leq \frac{N^{-\frac12}}{\epsilon^2}(2+ o(1)).\label{eq:Chebyshev}
\end{equation}

By going through the proof of Theorem \ref{thm:2ndmoment-O(1)}, one can check that the restriction on the maximal degree of $\Xi$ is only required for the case $\K=N$. In fact, it is optimal in the sense that we can construct a system  $\Xi\notin \mathscr{H}$ with maximal degree that grows like $\exp(e^{\frac12\log(2)N+o(N)})$ such that for $\K=N$ the variance blows up as $N\to\infty$, see Remark \ref{rem:blowup}.

For homogeneous systems $\Xi\in\mathscr{H}$, 
Shub and Smale \cite{ShubSmaleII} proved \eqref{eq:1stmoment}
for $\K=N$.\footnote{As they worked with roots in the projective space, their result and \eqref{eq:1stmoment} differ by a factor of $2$ accounting for the fact that every root in the projective space corresponds to two antipodal roots on the sphere.}
 Kostlan \cite{Kostlan1993}
computed earlier the expectation of $\sol_{N,\K}$ for any $\K\leq N$,
but only for $\Xi\in\mathfrak{\mathscr{H}}_{d}$. Aza\"{i}s and Wschebor
proved in \cite{AzaisWschebor} the same result as in \cite{ShubSmaleII}
using the Kac-Rice formula. Some extensions were also given in \cite{AzaisWschebor,EdelmanKostlan}. Although we could not find an explicit proof of Theorem \ref{thm:1stmoment} in earlier works, it easily follows from current standard techniques as in \cite{AzaisWschebor}, see the short proof in Section \ref{sec:1stmoment}. 
Finally, for $\K=N$ and $\Xi\in\mathfrak{\mathscr{H}}_{d}$ with
$d\geq3$ Wschebor proved in \cite{Wschebor} that the normalized variance 
as in (\ref{eq:momentsmatching}) vanishes as $N\to\infty$ (for $d=2$ he showed that the variance
is bounded by $1$ asymptotically). His result applies to the non-mixed
systems $\mathfrak{\mathscr{H}}_{d}$ of Kostlan \cite{Kostlan1993} but does not cover the mixed
homogeneous systems $\mathfrak{\mathscr{H}}$ studied by Shub and Smale, or the more general
non-homogeneous systems we consider. It also does not cover the case
that $\K<N$. Importantly, the mixed case $\Xi\in\mathscr{H}$ is the setting relevant to the real version of Smale's 17th problem.

In a foundational series of papers from the 90s,
\cite{ShubSmaleI,ShubSmaleII,ShubSmaleIII,ShubSmaleV,ShubSmaleIV}
Shub and Smale have studied the computational complexity of solving
systems of homogeneous polynomial equations. An important idea they
have put forward is that this problem needs to be analyzed from a
probabilistic point of view (in fact, \cite{ShubSmaleII,ShubSmaleV,ShubSmaleIV}
deal with probabilistic questions). In the real setting, the random systems they were interested
in are exactly $\vec{f}_{N}(\bx)$ above with $\Xi\in\mathscr{H}$. Their main focus, however, was on the complex setting in which the system is defined similarly, only with  $g_{i_{1},\dots,i_{p}}$
replaced by complex Gaussian variables.\footnote{To define a probability measure they took a more abstract approach
in \cite{ShubSmaleII} instead of the explicit definition (\ref{eq:random_poly}).
First they endow the space of homogeneous polynomials with complex
(resp. real) coefficients with a unitarily (resp. orthogonally) invariant
inner product (sometimes called the Bombieri-Weyl or Kostlan inner
product). This in turn induces a Riemannian structure and volume form
on the corresponding projective space of systems of homogeneous polynomials,
which normalized to $1$ gives a probability measure. The law of the
projection of $\vec{f}_{N}(\bx)$ (with $\Xi\in\mathscr{H}$) to the
projective space of homogeneous polynomial systems coincides with
the latter. See \cite{Kostlan1993,Kostlan2000} for more details.} Note that in the complex case, when $\K=N$ the question we study is trivial as $\sol_{N,N}=2\prod_{k=1}^{N}d_{N,k}$  almost surely, by the B\'{e}zout theorem. 
In 1998 Smale published in \cite{SmaleProblems} a list of `mathematical problems for the next century'. His 17th Problem which we mentioned above, asked whether
a zero of a random system of complex homogeneous polynomial equations can
be found approximately, on the average, in polynomial time with a
uniform algorithm?\footnote{For the precise definition of the terms ``uniform'' and ``approximately''
in this context, we refer the reader to the papers cited in this paragraph.} In \cite{ShubSmaleV} a non-uniform such algorithm was found, though not in a constructive way. Recently,
Smale's problem was resolved in a series of important breakthroughs by Beltr\'{a}n
and Pardo \cite{BeltranPardo2,BeltranPardo1}, B\"{u}rgisser and
Cucker \cite{BurgisserCucker} and Lairez \cite{Lairez1,Lairez2}.
The problem and those algorithms were concerned with complex homogeneous
systems. Smale has also raised in \cite{SmaleProblems} the question
of finding a root of a real random homogeneous system. As he pointed
out, this problem is more difficult. Indeed, no efficient algorithm is known in the real case. For  mixed systems, i.e. where the degrees of the random homogeneous polynomials are allowed to be different,
up to now the only the expectation $\E\sol_{N,\K}$ was known, but not the typical behavior of $\sol_{N,\K}$. To the best of our knowledge, it has not even been established that a solution exists with high probability.

Theorem \ref{thm:2ndmoment-O(1)} is a result about high-dimensions.
Its asymptotics is concerned with the behavior as the dimension of
the space on which the random functions are defined goes to infinity.
As the convergence is uniform, we have concentration, say in the form
of (\ref{eq:Chebyshev}) for a given $\epsilon$, as soon as the dimension
$N$ is sufficiently large irrespective of the choice of $\Xi$ or
$\K$. This asymptotics is very different from the one considered
when studying the zeros of a single polynomial by letting the degree
go to infinity, a topic that has been extensively studied for decades,
see e.g. the monographs \cite{Bharucha-ReidSambandham,Farahmand,ZerosAnalyticFunctions}.
This subject goes back to the works of Bloch and P\'{o}lya \cite{BlochPolya}, Littlewood and Offord \cite{LittlewoodOfford,LittlewoodOffordIII,LittlewoodOffordI,LittlewoodOffordII},
Kac \cite{Kac} and Rice \cite{Rice} starting in the 1930s. In the
context of systems of polynomials as we study in the current work, in addition to the asymptotics we
consider there is another type similar to the latter, in which the
parameter space is fixed. Namely, for non-mixed homogeneous systems
$\mathfrak{\mathscr{H}}_{d}$ with $N$ fixed, the asymptotics of
the size of the zero set was studied as the degree $d$ goes to infinity.
In this setting, Armentano et al. computed the asymptotics of the variance of $\sol_{N,\K}$
as $d\to\infty$ for $\K\leq N$ and proved that it satisfies a CLT
in \cite{Armentano1,Armentano2,Armentano3}. Letendre \cite{Letendre}
computed the variance for $\K<N$ under the same asymptotics. 

For non-homogeneous systems, one may consider the zero set in $\R^{N+1}$
without restricting to the sphere. Another possible extension of our
work is for systems of random multi-homogeneous polynomials, in which
case one needs to consider the zero set restricted to a product of
spheres. For multi-homogeneous systems and other related sparse systems,
the expected number of solutions was computed by Rojas \cite{Rojas}
and McLennan \cite{McLennan}. We believe that our methods can be
extended to show vanishing of the variance in both these settings
as well. This will be explored in future work. 

Deciding whether a \emph{deterministic} system of polynomials is solvable
is computationally hard if one lets the number of equations and variables
go to infinity. Any system can be reduced to a quadratic system by
adding extra variables and equations. Barvinok and Rudelson \cite{RudelsonBarvinok}
recently provided a sufficient condition for solvability of quadratic
systems. In Section 1.4 of their work, they compute the threshold
for the number of equations in a random setting up to which their
algorithm is able to detect that a solution exist in the quadratic
case. Our main result implies for large enough $N=\K$ for any system
$\Xi$ there exist many solutions, and suggests random systems as
a benchmark test for algorithms. In particular, finding a natural
algorithm for deterministic systems which is able to detect a solution
with high probability for any system $\Xi$ (say of bounded degree
uniformly in $N$) is an open problem.

Finally, we mention that in theoretical physics, random functions
$f_{N}(\bx)$ as in (\ref{eq:random_poly}) on the sphere are called
spherical mixed $p$-spin spin glass Hamiltonians (up to normalization).
 The Kac-Rice formula
was applied in the context of spherical spin glasses to compute the
expected number of critical points, see the works of Auffinger, Ben
Arous and {\v{C}}ern{\'y} \cite{ABA2,A-BA-C} and Fyodorov \cite{Fyo1}.
It was also used to prove its concentration by a second moment argument
\cite{2nd,SubagZeitouniConcentration} and to study extremely
large critical values by the author and Zeitouni \cite{pspinext}
and analyze the Gibbs measure by the author in \cite{geometryGibbs}
and with Ben Arous and Zeitouni in \cite{geometryMixed}.

In the next section we prove Theorem \ref{thm:1stmoment}.  In Section \ref{sec:Second-moment-log-scale},
also starting from the Kac-Rice formula, we first prove that the variance
as in (\ref{eq:momentsmatching}) is dominated by the contribution
coming from pairs of points on the sphere which are approximately
orthogonal. In Section \ref{sec:Matching-of-moments} we complete
the proof of Theorem \ref{thm:2ndmoment-O(1)} by a refined analysis
of the latter.

\section{\label{sec:1stmoment}First moment: proof of Theorem \ref{thm:1stmoment}}

Let $(E_{i})_{i\leq N}=(E_{i}(\bx))_{i\leq N}$ be some orthonormal
frame field on $\SN$. That is, for any $\bx\in\SN$, $E_{i}(\bx)$
is an orthonormal basis of the tangent space $T_{\bx}\SN$. For any
differentiable function $h:\SN\to\R$, $E_{i}h(\bx)$ are the corresponding
directional derivatives of $h(\bx)$. Denote $\nabla f_{N}^{(k)}(\bx):=(E_{i}f_{N}^{(k)}(\bx))_{i\leq N}$.
By an abuse of notation, define the vector valued random field $\bx\mapsto\vec{f}_{\K}(\bx):=(f_{N}^{(k)}(\bx))_{k=1}^{\M}$
and the matrix $\nabla\vec{f}_{\K}(\bx)\in\R^{\M\times N}$ by
\[
(\nabla\vec{f}_{\K}(\bx))_{kl}:=E_{l}f_{N}^{(k)}(\bx).
\]
By a variant of the Kac-Rice formula (see remark below) 
\begin{align}
\E\sol_{N,\K} & =\int_{\SN}\E\left[\J(\nabla\vec{f}_{\K}(\bx))\,\Big|\,\vec{f}_{\K}(\bx)=0\right]p_{\vec{f}_{\K}(\bx)}(0)d\HaussN(\bx),\label{eq:KR1stmoment1}
\end{align}
where $J(A)=\sqrt{\det AA^{T}}$, $\HaussN$ is the $N$-dimensional
Hausdorff measure on $\SN$ and $p_{\vec{f}_{\K}}(0)$ is the Gaussian
density of the vector $\vec{f}_{\K}(\bx)$ evaluated at the origin.
From independence 
\[
p_{\vec{f}_{\K}(\bx)}(0)=\prod_{k=1}^{\M}p_{f_{N}^{(k)}(\bx)}(0)=\prod_{k=1}^{\M}\frac{1}{\sqrt{2\pi\xi_{N,k}(1)}},
\]
where $p_{f_{N}^{(k)}(\bx)}(t)$ is the density of the centered Gaussian
variable $f_{N}^{(k)}(\bx)$ with variance $\xi_{N,k}(1)$. 
\begin{rem}
\label{rem:regularity}In our setting we can use the variant of the
Kac-Rice formula in Theorem 6.8 of \cite{AzaisWscheborBook}. It
cannot be applied directly, since there the parameter space of the
Gaussian process in the theorem is an open subset of $\R^{N+1}$ while we consider
the zero set on the unit sphere $\SN$. However, we may instead use
the process $\vec{g}_{\K}(\bx)=\vec{f}_{\K}(\bx/\|\bx\|)$ on a spherical
shell of width $\epsilon>0$ in $\R^{N+1}$. Each zero of $\vec{f}_{\K}(\bx)$
in $\SN$ then corresponds to a fiber of length $\epsilon$ of the
zero set of $\vec{g}_{\K}(\bx)$ in the shell, using which one can
easily verify (\ref{eq:KR1stmoment1}) from (6.27) of \cite{AzaisWscheborBook}.
Conditions (i)-(iii) in Theorem 6.8 of \cite{AzaisWscheborBook} follow easily from the definition
of $\vec{f}_{\K}(\bx)$. To verify condition (iv), it is enough to
prove the same condition with the Jacobian matrix of $\vec{g}_{\K}(\bx)$
replaced by $\nabla\vec{f}_{\K}(\bx)$ (since the rank of the former
is at least that of the latter). This condition follows from Proposition
6.12 of \cite{AzaisWscheborBook} and Lemma \ref{lem:gradient} below.
\end{rem}

\begin{lem}
\label{lem:gradient}For any $\bx\in\SN$, the elements $E_{l}f_{N}^{(k)}(\bx)$
of $\nabla\vec{f}_{\K}(\bx)$ are independent variables with distribution
$\N(0,\xi_{N,k}'(1))$ and $\nabla\vec{f}_{\K}(\bx)$ is independent
of $\vec{f}_{\K}(\bx)$.
\end{lem}

\begin{proof}
Since all variables are jointly Gaussian, the lemma follows by a computation
of covariances. Such calculation was carried out e.g. in \cite{ABA2},
but for completeness we provide a quick proof. For $\bv\in\R^{N+1}$
denote by $\frac{d}{d\bv}$ the corresponding directional derivative.
Let $f_{N}(\bx)$ be the centered Gaussian process as in (\ref{eq:random_poly})
with covariance function $\E f_{N}(\bx)f_{N}(\by)=\xi(\bx\cdot\by)$
for $\xi(t)=\sum_{p=2}^{d}a_{p}^{2}t^{p}$. Then, 
\begin{align*}
\E\frac{d}{d\bv}f_{N}(\bx)f_{N}(\bx') & =\frac{d}{d\bv}\xi(\bx\cdot\bx')=\xi'(\bx\cdot\bx')\bv\cdot\bx',\\
\E\frac{d}{d\bv}f_{N}(\bx)\frac{d}{d\bv'}f_{N}(\bx') & =\frac{d}{d\bv}\frac{d}{d\bv'}\xi(\bx\cdot\bx')=\xi'(\bx\cdot\bx')\bv\cdot\bv'+\xi''(\bx\cdot\bx')\left(\bv\cdot\bx'\right)\left(\bv'\cdot\bx\right).
\end{align*}

With $\bx=\bx'$, the first equation implies the independence of $\nabla\vec{f}_{\K}(\bx)$
and $\vec{f}_{\K}(\bx)$. For orthogonal unit vectors $\bv$, $\bv'$,
$\bx$, from the second equation with $\bx=\bx'$ the elements of
$\nabla\vec{f}_{\K}(\bx)$ are independent of each other. From the
same equation with $\bx=\bx'$ and $\bv=\bv'$ and $\bv$, $\bx$
orthogonal unit vectors, the variance of each element $E_{l}f_{\K}^{(k)}(\bx)$
of $\nabla\vec{f}_{\K}(\bx)$ is $\xi_{N,k}'(1)$.
\end{proof}
Combining the above with (\ref{eq:KR1stmoment1}), we have that for
arbitrary $\bx\in\SN$,
\begin{align}
\E\sol_{N,\K} & =\V(\SN)\E\J(\bM(\bx))\left(\frac{1}{2\pi}\right)^{\frac{\M}{2}}\prod_{k=1}^{\M}\sqrt{\frac{\xi_{N,k}'(1)}{\xi_{N,k}(1)}},\label{eq:KR1stmoment2}
\end{align}
where $\V(\SN)=\HaussN(\SN)$ is the volume of the unit sphere and
$\bM(\bx)=\bM_{N,\K}(\bx)$ is a $\K\times N$ matrix with i.i.d.
$\N(0,1)$ elements defined by
\begin{equation}
(\bM(\bx))_{kl}:=\frac{(\nabla\vec{f}_{\K}(\bx))_{kl}}{\sqrt{\xi_{N,k}'(1)}}=\frac{E_{l}f_{N}^{(k)}(\bx)}{\sqrt{\xi_{N,k}'(1)}}.\label{eq:Mx}
\end{equation}

In (\ref{eq:random_poly}) we start the summation from $p=2$, which
will be important in the proof of Theorem \ref{thm:2ndmoment-O(1)}.
However, everything in the current proof works also for $\xi_{N,k}(t)=\xi(t)=t$
with $f_{N}^{(k)}(\bx)=\bx\cdot G^{(k)}$ where $G^{(k)}$ are i.i.d.
Gaussian vectors whose elements are i.i.d. $\N(0,1)$ variables. For
this choice, $\sol_{N,\K}$ is the volume of the set of points on
the sphere which are orthogonal to all the vector $(G^{(k)})_{k=1}^{\K}$
and thus $\sol_{N,\K}=\V(\SNK)$ almost surely. It follows that 
\begin{equation}
\V(\SN)\E\J(\bM(\bx))\left(\frac{1}{2\pi}\right)^{\frac{\K}{2}}=\V(\SNK).\label{eq:MNdet}
\end{equation}
Since the law of $\bM(\bx)$ is independent of $\Xi$, by plugging this into
(\ref{eq:KR1stmoment2}), (\ref{eq:1stmoment}) follows for general
$\Xi$.\qed

\section{\label{sec:Second-moment-log-scale}Second moment on logarithmic
scale}

Given $I\subset[-1,1]$, consider the set
\[
\left\{ (\bx,\by)\in\SN\times\SN:\,\bx\cdot\by\in I,\,\vec{f}_{\K}(\bx)=\vec{f}_{\K}(\by)=0\right\} .
\]
If $\K=N$ define the random variable $\sol_{N,\K}^{(2)}(I)=\sol_{N,\K}^{(2)}(\Xi,I)$
as the cardinality of this set, and if $\K<N$ define it as the $2(N-\K)$-dimensional
Hausdorff measure of it. 

In this section we apply (a variant of) the Kac-Rice formula to $\sol_{N,\K}^{(2)}(I)$
and prove the following theorem weaker than Theorem \ref{thm:2ndmoment-O(1)}.
\begin{thm}
\label{thm:2ndmoment-log}Let $\alpha\in(0,\log 2)$
and $\epsilon>0$ and denote $I(\epsilon)=[-1,1]\setminus[-\epsilon,\epsilon]$.
Then, 
\begin{align}
\limsup_{N\to\infty}\sup_{\K\leq N}\sup_{\Xi}\frac{1}{N}\log\frac{\E\sol_{N,\K}^{(2)}(\Xi,I(\epsilon))}{(\E\sol_{N,\K}(\Xi))^{2}} & <0,\label{eq:2ndmoment-log}
\end{align}
where the supremum is either over all $\Xi$ such that $\max_{k\leq N}\deg(\xi_{N,k})\leq e^{e^{N\alpha}}$ or over $\mathscr{H}$.
\end{thm}

Consider the random field $(\bx,\by)\to(\vec{f}_{\K}(\bx),\vec{f}_{\K}(\by))$
on 
\[
T(I):=\left\{ (\bx,\by)\in\SN\times\SN:\,\bx\cdot\by\in I\right\} .
\]
$\sol_{N,\K}^{(2)}(I)$ is the number or volume of pairs $(\bx,\by)$
at which this random field is equal to $0\in\R^{2\K}$. Hence, for
an interval $I\subset(-1,1)$, an application of a variant of the
Kac-Rice formula yields,\footnote{This follows
by a similar argument to Remark \ref{rem:regularity} with Lemma \ref{lem:gradient}
replaced by (\ref{eq:2grads}).} 
\begin{equation}
\begin{aligned}\mathbb{E}\sol_{N,\K}^{(2)}(I) & =\prod_{k=1}^{\K}\xi_{N,k}'(1)\int_{T(I)}p_{\vec{f}_{\K}(\bx),\vec{f}_{\K}(\by)}(0)D(\bx\cdot\by)d\HaussN\times\HaussN,\end{aligned}
\label{eq:KR1}
\end{equation}
where for arbitrary points $\bx,\by\in\SN$ with $r=\bx\cdot\by$
we define\footnote{Note that $D(r)$ depends on the choice of $\bx$ and $\by$ only
through $\bx\cdot\by=r$ and does not depend on the choice of the
orthonormal frame field $E_{i}$.}
\begin{equation}
D(r):=\E\left[\J(\bM(\bx))\J(\bM(\by))\,\bigg|\,\vec{f}_{\K}(\bx)=\vec{f}_{\K}(\by)=0\right]\label{eq:D(r)}
\end{equation}
with $\bM(\bx)=\bM_{N,\K}(\bx)$  as defined in (\ref{eq:Mx}) and
where
\[
p_{\vec{f}_{\K}(\bx),\vec{f}_{\K}(\by)}(0)=\prod_{k=1}^{\K}\det(2\pi\Sigma_{N,k}(r))^{-\frac{1}{2}}=(2\pi)^{-\K}\prod_{k=1}^{\K}(\xi_{N,k}(1)^{2}-\xi_{N,k}(r)^{2})^{-\frac{1}{2}}
\]
is the density of $(\vec{f}_{\K}(\bx),\vec{f}_{\K}(\by))$ at the
origin and 
\[
\Sigma_{N,k}(r)=\left(\begin{array}{cc}
\xi_{N,k}(1) & \xi_{N,k}(r)\\
\xi_{N,k}(r) & \xi_{N,k}(1)
\end{array}\right).
\]

Since the integrand in (\ref{eq:KR1}) only depends on $\bx\cdot\by$,
for an arbitrary $\bx\in\SN$,
\[
\mathbb{E}\sol_{N,\K}^{(2)}(I)=\V(\SN)\prod_{k=1}^{\K}\xi_{N,k}'(1)\int_{\{\by\in\SN:\,\bx\cdot\by\in I\}}p_{\vec{f}_{\K}(\bx),\vec{f}_{\K}(\by)}(0)D(\bx\cdot\by)d\HaussN(\by).
\]
Next, using the co-area formula with the function $\rho(\by)=\bx\cdot\by$,
we may express the integral w.r.t. $\by$ as a one-dimensional integral
over a parameter $r$ (the volume of the inverse-image $\rho^{-1}(r)$
and the inverse of the Jacobian are given by $\V(\SNt)(1-r^{2})^{\frac{N-1}{2}}$
and $(1-r^{2})^{-\frac{1}{2}}$, respectively). This yields 
\begin{equation}
\begin{aligned}\mathbb{E}\sol_{N,\K}^{(2)}(I) & =C_{N,\K}\int_{I}\prod_{k=1}^{\K}\left(\frac{\xi_{N,k}'(1)^{2}}{\xi_{N,k}(1)^{2}-\xi_{N,k}(r)^{2}}\right)^{\frac{1}{2}}(1-r^{2})^{\frac{N-2}{2}}D(r)dr,\end{aligned}
\label{eq:KR2}
\end{equation}
where
\begin{equation}
C_{N,\K}=\V(\SN)\V(\SNt)(2\pi)^{-\K}.\label{eq:CN}
\end{equation}
We will prove Theorem \ref{thm:2ndmoment-log} in Subsection \ref{subsec:pf_thm_2ndmoment-log}
below, after we prove several auxiliary results.

We recall that $f_{N}(\bx)$ is defined in (\ref{eq:random_poly})
and that $\E f_{N}(\bx)f_{N}(\by)=\xi(\bx\cdot\by)$ for $\xi(t)=\sum_{p=2}^{d}a_{p}^{2}t^{p}$. 
\begin{lem}
\label{lem:conditionalGradient}Let $\bx,\by\in\SN$ be two points
with $\bx\cdot\by=r\in(-1,1)$. Denote by $\E_{0}$ expectation conditional
on $f_{N}(\bx)=f_{N}(\by)=0$. Then there exists a choice of the frame
field $(E_{i})$ for which, conditional on $f_{N}(\bx)=f_{N}(\by)=0$,
$\nabla f_{N}(\bx)$ and $\nabla f_{N}(\by)$ are centered jointly
Gaussian matrices such that for any $i\neq j$ and $\bz_{1},\bz_{2}\in\{\bx,\by\}$,
\[
\E_{0}\left\{ E_{i}f_{N}(\bz_{1})E_{j}f_{N}(\bz_{2})\right\} =0,
\]
and for any $i$,
\begin{align*}
\E_{0}\left\{ E_{i}f_{N}(\bx)E_{i}f_{N}(\bx)\right\}  & =\E_{0}\left\{ E_{i}f_{N}(\by)E_{i}f_{N}(\by)\right\} \\
 & =\begin{cases}
\xi'(1) & i<N\\
\xi'(1)-\frac{\xi'(r)^{2}(1-r^{2})}{\xi(1)^{2}-\xi(r)^{2}}\xi(1) & i=N
\end{cases}
\end{align*}
and 
\begin{align*}
\E_{0}\left\{ E_{i}f_{N}(\bx)E_{i}f_{N}(\by)\right\}  & =\begin{cases}
\xi'(r) & i<N\\
r\xi'(r)-\xi''(r)\left(1-r^{2}\right)-\frac{\xi'(r)^{2}(1-r^{2})}{\xi(1)^{2}-\xi(r)^{2}}\xi(r) & i=N.
\end{cases}
\end{align*}
\end{lem}

\begin{proof}
By rotational invariance, we may assume that $\bx=(0,\ldots,0,1)$
and $\by=(0,\ldots,\sqrt{1-r^{2}},r)$. By Lemma A.1 of \cite{geometryMixed},
there exists a choice of the frame field $(E_{i})$ such that
\begin{align*}
\mbox{\ensuremath{\mathbb{E}}}\left\{ f_{N}(\bx)E_{i}f_{N}(\bx)\right\}  & =\mbox{\ensuremath{\mathbb{E}}}\left\{ f_{N}(\by)E_{i}f_{N}(\by)\right\} =0,\\
\mbox{\ensuremath{\mathbb{E}}}\left\{ f_{N}(\by)E_{i}f_{N}(\bx))\right\}  & =-\mbox{\ensuremath{\mathbb{E}}}\left\{ f_{N}(\bx)E_{i}f_{N}(\by))\right\} =\xi'(r)\left(1-r^{2}\right)^{1/2}\delta_{i=N}
\end{align*}
and 
\begin{equation}
\begin{aligned}\mbox{\ensuremath{\mathbb{E}}}\left\{ E_{i}f_{N}(\bx)E_{j}f_{N}(\bx)\right\}  & =\mbox{\ensuremath{\mathbb{E}}}\left\{ E_{i}f_{N}(\by)E_{j}f_{N}(\by)\right\} =\xi'(1)\delta_{i=j},\\
\mbox{\ensuremath{\mathbb{E}}}\left\{ E_{i}f_{N}(\bx)E_{j}f_{N}(\by)\right\}  & =\left[r\xi'(r)-\xi''(r)\left(1-r^{2}\right)\right]\delta_{i=j=N}+\xi'(r)\delta_{i=j\neq N}.
\end{aligned}
\label{eq:2grads}
\end{equation}

Note that the inverse of the covariance matrix of $(f_{N}(\bx),f_{N}(\by)$)
is
\[
\Sigma^{-1}:=\left(\begin{array}{cc}
\xi(1) & \xi(r)\\
\xi(r) & \xi(1)
\end{array}\right)^{-1}=\frac{1}{\xi(1)^{2}-\xi(r)^{2}}\left(\begin{array}{cc}
\xi(1) & -\xi(r)\\
-\xi(r) & \xi(1)
\end{array}\right).
\]
The lemma thus follows from the well-known formulas for the conditional distribution of Gaussian
variables.
%
\end{proof}
Next we prove the following upper bound using (\ref{eq:KR2}) and
Lemma \ref{lem:conditionalGradient}.
\begin{lem}
\label{lem:general_bd}For any interval $I\subset(-1,1)$,
\begin{equation}
\begin{aligned}\mathbb{E}\sol_{N,\K}^{(2)}(I) & \leq C_{N,\K}\frac{N!}{(N-\K)!}\int_{I}\prod_{k=1}^{\K}\bigg(\frac{\xi_{N,k}'(1)^{2}(1-r^{2})}{\xi_{N,k}(1)^{2}-\xi_{N,k}(r)^{2}}\bigg)^{\frac{1}{2}}(1-r^{2})^{\frac{N-\K-2}{2}}\Phi_{N,\K}(r)dr,\end{aligned}
\label{eq:KR2-1}
\end{equation}
where $C_{N,\K}$ is defined in (\ref{eq:CN}) and $\Phi_{N,\K}(r)\in[0,1]$
is given by 
\begin{equation}
\begin{aligned}\Phi_{N,\K}(r) & :=\frac{1}{N}\bigg[N-K+\sum_{k=1}^{\K}\bigg(1-\frac{\xi_{N,k}(1)}{\xi_{N,k}'(1)}\frac{\xi_{N,k}'(r)^{2}(1-r^{2})}{\xi_{N,k}(1)^{2}-\xi_{N,k}(r)^{2}}\bigg)\bigg].\end{aligned}
\label{eq:Phi}
\end{equation}
\end{lem}

\begin{proof}
Suppose that $\bx,\by\in\SN$ are two points such that $\bx\cdot\by=r$.
Recall that the $k$-th row of $\nabla\vec{f}_{\K}(\bx)$ is $\nabla f_{N}^{(k)}(\bx)$.
For different values of $k$, 
\[
\left(\nabla f_{N}^{(k)}(\bx),\nabla f_{N}^{(k)}(\by),f_{N}^{(k)}(\bx),f_{N}^{(k)}(\by)\right)
\]
are independent from each other. Let $\tilde{\bM}$ be a random matrix
whose law is defined as the conditional law of $\bM(\bx)$ (defined
in (\ref{eq:Mx})) given that$\vec{f}_{\K}(\bx)=\vec{f}_{\K}(\by)=0$.
By Lemma \ref{lem:conditionalGradient}, for an appropriate choice
of $E_{i}$, the $k$-th row of $\tilde{\bM}$ is equal in distribution
to
\begin{equation}
\big(\tilde{\bM}_{kj}\big)_{j=1}^{N}=\Big(W_{1},\ldots,W_{N-1,}\sqrt{\lambda^{(k)}(r)}W_{N}\Big),\label{eq:Mtilde-1}
\end{equation}
where $W_{j}$ are i.i.d. $\N(0,1)$ variables and 
\begin{equation}
\lambda^{(k)}(r)=1-\frac{\xi_{N,k}(1)}{\xi_{N,k}'(1)}\frac{\xi_{N,k}'(r)^{2}(1-r^{2})}{\xi_{N,k}(1)^{2}-\xi_{N,k}(r)^{2}}\geq0.\label{eq:lambdak}
\end{equation}
By symmetry, the conditional law of $\bM(\by)$ is also given by (\ref{eq:Mtilde-1}).

By Cauchy-Schwarz, 
\begin{equation}
\begin{aligned}D(r) & \leq\E\left[\J(\bM(\bx))^{2}\,\big|\,\vec{f}_{\K}(\bx)=\vec{f}_{\K}(\by)=0\right]=\E\left[\J(\tilde{\bM})^{2}\right]\end{aligned}
.\label{eq:D(r)_Bd}
\end{equation}
Denote the last column of $\tilde{\bM}$ by $Z$. Let $\mathbf{Q}$
be a $\K\times\K$ orthogonal matrix, measurable w.r.t. $Z$, such
that
\begin{equation}
\mathbf{Q}\tilde{\bM}=\left(\begin{array}{cc}
\tilde{\mathbf{A}} & 0\\
V & \|Z\|
\end{array}\right),\label{eq:QM}
\end{equation}
for some $\K-1\times N-1$ matrix $\tilde{\mathbf{A}}$ and row vector
$V$ of length $N-1$. Note that $\J(\tilde{\bM})=\J(\mathbf{Q}\tilde{\bM})$
and that the elements of $\tilde{\mathbf{A}}$ and $V$ are i.i.d.
$\N(0,1)$ variables which are also independent of $\|Z\|$. For any
$\K\times N$ matrix $\mathbf{B}$, $\J(\mathbf{B})=\prod_{i=1}^{\K}\Theta_{i}(\mathbf{B})$
where we define $\Theta_{i}(\mathbf{B})$ as the norm of projection
of the $i$-th row of $\mathbf{B}$ onto the orthogonal space to the
span of all previous rows (and by convention, $\Theta_{1}(\mathbf{B})$
as the norm of the first row). 

Applying the latter identity to $\mathbf{Q}\tilde{\bM}$, one sees
that in distribution 
\[
\J(\tilde{\bM})=\prod_{k=1}^{\K-1}\chi_{N-k}\cdot\sqrt{\|Z\|^{2}+\chi_{N-\K}^{2}},
\]
where $\chi_{i}$ denotes a Chi variable with $i$ degrees of freedom
and $\chi_{i}$ are independent of each other and $Z$. Hence,
\begin{equation}
0\leq\E\left[\J(\tilde{\bM})^{2}\right]\leq\frac{(N-1)!}{(N-\K)!}(N-\K+\sum_{k=1}^{\K}\lambda^{(k)}(r))\leq\frac{N!}{(N-\K)!}.\label{eq:JMtilde}
\end{equation}
Combining the above with (\ref{eq:KR2}) yields (\ref{eq:KR2-1}).
\end{proof}
For $\K=N$, the next two corollaries easily follow from Lemma \ref{lem:general_bd}. 
\begin{cor}
\label{cor:log_bd}For some universal constant $c>0$, for any $\epsilon\in(0,1)$
and interval $I\subset(-1+\epsilon,1)$, 
\[
\mathbb{E}\sol_{N,N}^{(2)}(I)\leq\frac{cN}{\epsilon}\log\big(\max_{k\leq N}\deg(\xi_{N,k})\big)\sup_{r\in I}\prod_{k=1}^{N}\Big(\frac{\xi_{N,k}'(1)^{2}(1-r^{2})}{\xi_{N,k}(1)^{2}-\xi_{N,k}(r)^{2}}\Big)^{\frac{1}{2}}.
\]
\end{cor}

\begin{proof}
Using that $\V(\SN)=\frac{2\pi^{\frac{N+1}{2}}}{\Gamma\left(\frac{N+1}{2}\right)}$
and Stirling's approximation, one can verify that $C_{N,N}N!\leq cN$,
for some universal constant $c>0$. Hence, from Lemma \ref{lem:general_bd},
to prove the corollary it will be enough to show that 
\begin{equation}
\int_{-1+\epsilon}^{1}\frac{1}{1-r^{2}}\Phi_{N,N}(r)dr\leq\frac{10}{\epsilon}\log\big(\max_{k\leq N}\deg(\xi_{N,k})\big).\label{eq:req}
\end{equation}

Suppose that $\xi(t)=\sum_{p=2}^{d}a_{p}^{2}t^{p}$ for some $a_{p}\geq0$
and define $\lambda(r)$ as in (\ref{eq:lambdak}) with $\xi_{N,k}$
replaced by $\xi$. For any $r\in(-1,1)$,
\[
0<\xi(1)^{2}-\xi(r)^{2}\leq(1-r)\max_{t\in[r,1]}\frac{d}{dt}\Big(\xi(t)^{2}\Big)=2\xi(1)\xi'(1)(1-r)
\]
and
\begin{equation}\label{eq:lmbdabd2}
\lambda(r)\leq1-\frac{\xi'(r)^{2}(1+r)}{2\xi'(1)^{2}}=\frac{\xi'(1)^{2}-\xi'(r)^{2}+\xi'(1)^{2}-r\xi'(r)^{2}}{2\xi'(1)^{2}}.
\end{equation}
Using $\xi(t)=\sum_{p=2}^{d}a_{p}^{2}t^{p}$ we obtain that 
\begin{equation}\label{eq:lambdabd}
\frac{\lambda(r)}{1-r^{2}}\leq\frac{\sum_{p,q=2}^{d}a_{p}^{2}a_{q}^{2}pq\left(\sum_{i=0}^{p+q-3}r^{i}+\sum_{i=0}^{p+q-2}r^{i}\right)}{2(1+r)\sum_{p,q=2}^{d}a_{p}^{2}a_{q}^{2}pq}.
\end{equation}
Hence, 
\[
\int_{-1+\epsilon}^{1}\frac{\lambda(r)}{1-r^{2}}dr\leq\frac{2}{\epsilon}\int_{0}^{1}\sum_{i=0}^{2d-2}r^{i}dr\leq\frac{2}{\epsilon}(\log2d+1)\leq\frac{10}{\epsilon}\log d.
\]
Since $\Phi_{N,N}(r)=\frac{1}{N}\sum_{k=1}^{N}\lambda^{(k)}(r)$ (see
(\ref{eq:lambdak})), this implies (\ref{eq:req}).
\end{proof}

\begin{cor}
	\label{cor:homogeneous}For some universal constant $c'>0$, if $\Xi\in\mathscr{H}$ then for any interval $I\subset(-1,1)$, 
	\[
	\mathbb{E}\sol_{N,N}^{(2)}(I)\leq c' \cdot   \sup_{r\in I}\sum_{j=1}^N\bigg[ \frac{\xi_{N,j}'(1)}{\xi_{N,j}(1)} \prod_{k\in\{1,\ldots,N\}\setminus\{j\}}\Big(\frac{\xi_{N,k}'(1)^{2}(1-r^{2})}{\xi_{N,k}(1)^{2}-\xi_{N,k}(r)^{2}}\Big)^{\frac{1}{2}}\bigg].
	\]
\end{cor}

\begin{proof}
	As stated in the previous proof, $C_{N,N}N!\leq cN$ for some universal $c$. Using Lemma \ref{lem:general_bd}, to prove the current lemma it will be enough to show that for any $k$ and $r\in (-1,1)$,
	\begin{equation}
	\frac{1}{1-r^{2}}\lambda^{(k)}(r) \Big(\frac{1-r^{2}}{1-\xi_{N,k}(r)^{2}/\xi_{N,k}(1)^{2}}\Big)^{\frac{1}{2}} \leq \theta(r),\label{eq:req2}
	\end{equation}
where $\theta(r)$ is some universal function such that $\int_{-1}^{1}\theta(r)dr<\infty$, assuming that $\xi_{N,k}(t)=t^{d_{N,k}}$ for some $d_{N,k}$. Indeed, using \eqref{eq:lmbdabd2}, the left-hand side of \eqref{eq:req2} is bounded by
\[
\frac{1-\frac12(1+r)r^{2d_{N,k}-2}}{1-r^2}
\bigg(\frac{1-r^2}{1-r^{2d_{N,k}}}\bigg)^{\frac{1}{2}}
\leq \sqrt{\frac{1}{1-r^2}}.
\]
\end{proof}
We will need the following lemma to deal with inner-product values
close to $-1$, here again for the case $\K=N$.
\begin{lem}
\label{lem:negativeoverlaps} For any $\epsilon<1/4$, with $I_{-}(\epsilon)=[-1,-1+\epsilon]$
and $I_{+}(\epsilon)=[1-\epsilon,1]$,
\[
\sol_{N,N}^{(2)}(I_{-}(\epsilon))\leq2\sol_{N,N}^{(2)}(I_{+}(4\epsilon)).
\]
\end{lem}

\begin{proof}
Define $A=\{\bx\in\SN:\,\vec{f}_{N}(\bx)=0\}$. Let $B$ be a maximal
set with the following property: its elements are disjoint subsets
of $A$ of size $2$ and if $\{\bx,\by\}\in B$ then $\bx\cdot\by\leq-1+\epsilon$.
Define $B_{1}=\cup_{\{\bx,\by\}\in B}\{\bx,\by\}$. From maximality,
if $(\bx,\by)\in A\times A$ and $\bx\cdot\by\leq-1+\epsilon$ then
$|\{\bx,\by\}\cap B_{1}|\geq1$. Hence,
\begin{align*}
 & \#\{(\bx,\by)\in A\times A:\,\bx\cdot\by\leq-1+\epsilon\}\\
 & \leq2\#\{(\bx,\by)\in A\times B_{1}:\,\bx\cdot\by\leq-1+\epsilon\}\\
 & =2\sum_{\by\in B_{1}}\#\{\bx\in A:\,\bx\cdot\by\leq-1+\epsilon\}\leq\cdots
\end{align*}

For any $\by\in B_{1}$ there is exactly one element $\{\by,\by_{*}\}\in B$
containing $\by$. If $\bx\cdot\by\leq-1+\epsilon$, then $\|\bx+\by\|^{2}\leq2\epsilon$.
The same holds for $\bx=\by_{*}$. Therefore, by the triangle inequality,
$\|\bx-\by_{*}\|\leq2\sqrt{2\epsilon}$ and $\bx\cdot\by_{*}\geq1-4\epsilon$.
Thus, we can continue the chain of inequalities as follows
\begin{align*}
\cdots & \leq2\sum_{\by_{*}\in B_{1}}\#\{\bx\in A:\,\bx\cdot\by_{*}\geq1-4\epsilon\}\\
 & \leq2\#\{(\bx,\by)\in A\times A:\,\bx\cdot\by\geq1-4\epsilon\}.
\end{align*}
\end{proof}

\begin{rem}\label{rem:blowup}
	Using that $\Phi_{N,\K}(r)\leq1$ in \eqref{eq:KR2-1} gives a bound which is not sufficiently good for $\K=N$, as the integrand behaves like $C/(1-r^2)$ close to $r=1$, which is not an integrable function. 
	The last three results will be used to deal with this issue, under the assumptions on $\Xi$ as in Theorem \ref{thm:2ndmoment-O(1)}. 
	
	Assume that $\K=N$ and that $\xi_{N,1}(r)=\xi(r)=r^2+(\log p/p) r^p$ and that for $k\geq 2$, $\xi_{N,k}(r)=r^2$. 
	Denote $M_k(r)=\big(\frac{1-r^{2}}{1-\xi_{N,k}(r)^{2}/\xi_{N,k}(1)^{2}}\big)^{\frac{1}{2}}$  and $L(r)=(1-r^{2})^{-1}\Phi_{N,N}(r)$. Then, $\prod_{k=2}^{N}M_k(r)\geq 2^{-(N-1)/2}$ uniformly in $r\in(-1,1)$ for some $C>0$. One can check that $M_1(r)L(r)\geq (N4(1-r^2))^{-1}$ for any $r\in(\frac12,1-2\log p/p)$, for large enough $p$. Using this one can verify that if we take $p=p(N)=\lfloor \exp(e^{\frac12\log(2)N+g(N)})\rfloor$ for $g(N)=o(N)$ that goes to $\infty$, then the right-hand side of \eqref{eq:KR2-1} divided by the first moment squared goes to $\infty$ roughly like $e^{g(N)}$ as $N\to\infty$.
	
	We recall that the bound of \eqref{eq:KR2-1} was derived using Cauchy-Schwarz in \eqref{eq:D(r)_Bd}. However, using Lemma \ref{lem:GaussianBound} one can show that $D(r)\geq \left[\E \J(\tilde{\bM})\right]^2$ which asymptotically gives a matching lower bound to \eqref{eq:KR2-1}, up to a polynomial factor in $N$. Hence, the bound we derived for \eqref{eq:KR2-1} also gives a similar bound for $\mathbb{E}[\sol_{N,N}^{2}]/[\mathbb{E}\sol_{N,N}]^{2}$, showing it goes to $\infty$ for the system as above. Thus, our assumption on the maximal degree in the  non-homogeneous case in Theorem \ref{thm:2ndmoment-O(1)} is essentially optimal.
\end{rem}

\subsection{\label{subsec:pf_thm_2ndmoment-log}Proof of Theorem \ref{thm:2ndmoment-log}}

Let $\epsilon>0$ be an arbitrary number and denote $I(\epsilon)=[-1,1]\setminus[-\epsilon,\epsilon]$.
Let $\xi(r)=\sum_{p=2}^{d}a_{p}^{2}r^{p}$ and define $\nu(r):=\frac{\xi(r)}{\xi(1)}=\sum_{p=2}^{d}b_{p}^{2}r^{p}$,
where $b_{p}^{2}:=a_{p}^{2}/\sum_{p=2}^{d}a_{i}^{2}$. For any $r\in[-1,1],$
\[
1-\nu(r)^{2}=\sum_{p=2}^{d}\sum_{q=2}^{d}b_{p}^{2}b_{q}^{2}(1-r^{p+q})\geq1-r^{4}
\]
and 
\begin{equation}
\frac{\xi'(1)^{2}(1-r^{2})}{\xi(1)^{2}-\xi(r)^{2}}/\frac{\xi'(1)^{2}}{\xi(1)^{2}}=\frac{1-r^{2}}{1-\nu(r)^{2}}\leq\frac{1-r^{2}}{1-r^{4}}=\frac{1}{1+r^{2}}.\label{eq:ub}
\end{equation}

Assume that $\K\leq N-1$. The random variable $\sol_{N,\K}^{(2)}(\{\pm1\})$
is bounded by twice the $2(N-\K)$-dimensional Hausdorff measure of
(\ref{eq:zeroset}). The Hausdorff dimension of the latter subset
is $N-\K$ almost surely since $\E\sol_{N,\K}<\infty$. Thus, $\E\sol_{N,\K}^{(2)}(\{\pm1\})=0$.
Using the additivity of $\sol_{N,\K}^{(2)}(\cdot)$, Lemma \ref{lem:general_bd},
Theorem \ref{thm:1stmoment} and (\ref{eq:ub}) we obtain that 
\begin{equation}
\begin{aligned}\frac{\E\sol_{N,\K}^{(2)}(I(\epsilon))}{(\E\sol_{N,\K})^{2}} & \leq\frac{C_{N,\K}N!}{\V(\SNK)^{2}(N-\K)!}\theta(1+\epsilon^{2})^{-\frac{\K}{2}}(1-\epsilon^{2})^{\frac{N-\K-1}{2}},\end{aligned}
\label{eq:bd1}
\end{equation}
where $\theta:=\int_{-1}^{1}1/\sqrt{1-r^{2}}dr<\infty$. Using the
Legendre duplication formula one can check that the ratio on the right-hand
side is bounded from above by $\beta N$ for some constant $\beta>0$,
uniformly in $N$ and $\K$. Hence, 
\[
\limsup_{N\to\infty}\sup_{\K\leq N-1}\sup_{\Xi}\frac{1}{N}\log\frac{\E\sol_{N,\K}^{(2)}(I(\epsilon))}{(\E\sol_{N,\K})^{2}}\leq\frac{1}{4}\max\{\log(1-\epsilon^{2}),-\log(1+\epsilon^{2})\}<0,
\]
where here the supremum can be taken over all $\Xi$ without any restriction.

Next, we move to the case $\K=N$. Since $I\mapsto\sol_{N,N}^{(2)}(I)$ is additive, it is enough to prove that 
\begin{equation}\label{eq:Z2ineq}
\limsup_{N\to\infty}\sup_{\Xi}\frac{1}{N}\log\frac{\E\sol_{N,N}^{(2)}(I_j)}{(\E\sol_{N,N})^{2}}<0,
\end{equation}
for some finite collection of subsets $I_1,\ldots,I_n$ whose union is $I(\epsilon)$. For $I=\{\pm1\}$ the inequality follows by Theorem \ref{thm:1stmoment}. If we take the supremum over $\Xi\in\mathscr{H}$, with $I=(-1,1)$ the inequality follows by Corollary \ref{cor:homogeneous} similarly to the above. 

Now assume that $\alpha\in(0,\frac12\log 2)$ and that the supremum is over all $\Xi$ such that $d_{N}(\Xi):=\max_{k\leq N}\deg(\xi_{N,k})\leq  e^{e^{N\alpha}}$. Choose some $\epsilon>0$ such that $\alpha<\frac12\log(1+(1-4\epsilon)^2)$. For $I=[-1+\epsilon,1-4\epsilon]$, \eqref{eq:Z2ineq} follows by Theorem \ref{thm:1stmoment}, Lemma \ref{lem:general_bd}, (\ref{eq:ub})  and the fact that $\Phi_{N,N}(r)\leq 1$ since on the same interval $1-r^2>c$ for some constant $c=c(\epsilon)$. For $I=[1-4\epsilon,1)$, from
Corollary \ref{cor:log_bd} and (\ref{eq:ub}),
\[
\mathbb{E}\sol_{N,N}^{(2)}(I)\leq\frac{cN}{\delta}\log d_{N}(\Xi)\Big(\frac{1}{1+(1-4\epsilon)^2}\Big)^{\frac{N}{2}}\prod_{k=1}^{N}\frac{\xi_{N,k}'(1)}{\xi_{N,k}(1)}.
\]
Thus, \eqref{eq:Z2ineq} follows for the same interval $I$. 
By Lemma \ref{lem:negativeoverlaps}, 
$
\mathbb{E}\sol_{N,N}^{(2)}([-1,-1+\epsilon)\leq2\mathbb{E}\sol_{N,N}^{(2)}([1-4\epsilon,1])
$, and \eqref{eq:Z2ineq} thus follows with $I=[-1,-1+\epsilon]$.\qed

\section{\label{sec:Matching-of-moments}Matching of moments: proof of Theorem
\ref{thm:2ndmoment-O(1)}}

Note that
$\sol_{N,\K}^{2}=\sol_{N,\K}^{(2)}([-1,1])$. From Theorem \ref{thm:2ndmoment-log},
there exists some sequence $\epsilon_{N}\to0$ such that
\[
\sup_{\K\leq N}\sup_{\Xi} \frac{\E\sol_{N,\K}^{(2)}(\Xi,[-1,1]\setminus(-\epsilon_{N},\epsilon_{N}))}{\E(\sol_{N,\K}(\Xi)^{2})}\leq O\Big(\frac{1}{N^2}\Big).
\]
Hence, to prove (\ref{eq:momentsmatching}) it will be enough to show
that 
\begin{equation}
\sup_{\K\leq N}\sup_{\Xi}\frac{\E\sol_{N,\K}^{(2)}(\Xi,(-\epsilon_{N},\epsilon_{N}))}{(\E\sol_{N,\K}(\Xi))^{2}}\leq 1 +2N^{-\frac12}+ o\big(N^{-\frac12}\big).\label{eq:lim1}
\end{equation}

By \cite[Eq. (2.2)]{MR2611044},
\[
\frac{\V(\SNt)}{\V(\SN)}=\frac{1}{\sqrt{\pi}}\frac{\Gamma((N+1)/2)}{\Gamma(N/2)}<\sqrt{\frac{N}{2\pi}}.
\]
Combining this with (\ref{eq:KR1stmoment2}) and (\ref{eq:KR2}), we obtain that
the ratio in (\ref{eq:lim1}) is bounded by 
\begin{equation}
\sqrt{\frac{N}{2\pi}}\int_{-\epsilon_{N}}^{\epsilon_{N}}\frac{D(r)}{(\E\J(\bM(\bx)))^{2}}(1-r^{2})^{\frac{N}{2}}\prod_{k=1}^{\K}\Big(1-\frac{\xi_{N,k}(r)^{2}}{\xi_{N,k}(1)^{2}}\Big)^{-\frac{1}{2}}dr,\label{eq:lim2}
\end{equation}
where $\bx\in\SN$ is arbitrary and $D(r)$ and $\bM(\bx)$ are defined
in (\ref{eq:D(r)}) and (\ref{eq:Mx}).

For any $r\in(-1,1)$, $\Xi$, $N$ and $\K$,
\[
\prod_{k=1}^{\K}\Big(1-\frac{\xi_{N,k}(r)^{2}}{\xi_{N,k}(1)^{2}}\Big)^{-\frac{1}{2}}\leq\Big(1-r^{4}\Big)^{-\frac{\K}{2}}\leq\Big(1-r^{4}\Big)^{-\frac{N}{2}}.
\]
For any $t\in\R$, $1+t\leq e^t$. For $t\in(0,1/2)$, $1/(1-t)\leq e^{2t}$.
Thus, assuming the lemma below, for large $N$, (\ref{eq:lim2}) is bounded from above
by
\[
\frac{N}{N-1}\sqrt{\frac{N}{2\pi}}\int_{-\epsilon_{N}}^{\epsilon_{N}}\exp\left(-\frac{Nr^{2}}{2}+Nr^4+2\sqrt{N}r^{2}\right)dr\leq
\frac{N}{N-1}\sqrt{\frac{1}{2\pi}}\int_{-\epsilon_{N}N^{\frac12}}^{\epsilon_{N}N^{\frac12}}\exp\left(-\frac{r^{2}}{2}+\frac{r^4}{N}+\frac{2r^2}{\sqrt N}\right)dr.
\]
The contribution to the integral from the complement of $[-c\log N,c\log N]$ is bounded by $1/N^2$ for a large enough constant $c>0$. Thus, it is enough to bound the expression above with the integral only over $[-c\log N,c\log N]$. For the latter, we may remove the $r^4/N$ term and multiply by the integral $1+O(\log N^4/N)$ to get an upper bound. Finally, by a change of variables we may move to the standard Gaussian density, where we need to multiply by the Jacobian $1/\sqrt{1-4/\sqrt{N}}$.
This proves (\ref{eq:lim1}) and thus (\ref{eq:momentsmatching}). It remains to prove the following lemma.
\begin{lem}
\label{lem:Dbar}As $N\to\infty$, uniformly in $r\in(-1,1)$ and $K\leq N$ and any $\Xi$, 
\begin{equation}
\frac{D(r)}{(\E\J(\bM(\bx)))^{2}}\leq 1+\frac{1}{N-1}+r^2\sqrt{ \frac{N\pi} 2}(1+o(1)).\label{eq:Dbar-1}
\end{equation}
\end{lem}

\begin{proof}
Recall the definition of $D(r)$ in (\ref{eq:D(r)}). By Lemma \ref{lem:conditionalGradient},
\[
D(r)=\E\left[\J(\bM^{(1)}(r))\J(\bM^{(2)}(r))\right],
\]
where $\bM^{(1)}(r)$ and $\bM^{(2)}(r)$ are jointly Gaussian, marginally
each has the same distribution as $\tilde{\bM}$ defined in (\ref{eq:Mtilde-1}),
and the covariance between the elements of $\bM^{(1)}(r)$ and $\bM^{(2)}(r)$
is only between elements with the same indices and is given by 
\begin{align}
 & \E\left\{ \bM_{kj}^{(1)}(r)\bM_{kj}^{(2)}(r)\right\} \nonumber \\
 & =\begin{cases}
\frac{\xi_{N,k}'(r)}{\xi_{N,k}'(1)} & j<N\\
r\frac{\xi_{N,k}'(r)}{\xi_{N,k}'(1)}-\frac{\xi_{N,k}''(r)}{\xi_{N,k}'(1)}(1-r^{2})-\frac{\xi_{N,k}'(r)^{2}}{\xi_{N,k}'(1)}\frac{(1-r^{2})}{\xi_{N,k}(1)^{2}-\xi_{N,k}(r)^{2}}\xi_{N,k}(r) & j=N.
\end{cases}\label{eq:cov}
\end{align}
Define $\hat{\bM}^{(1)}(r)$ as the matrix obtained from $\bM^{(1)}(r)$
by adding, for each $k$, to $\bM_{kN}^{(1)}(r)$ a centered Gaussian
variable independent of everything else with variance $1-\lambda^{(k)}(r)$
(see (\ref{eq:lambdak})) so that the elements of $\hat{\bM}^{(1)}(r)$
are i.i.d. $\N(0,1)$ variables. Define $\hat{\bM}^{(2)}(r)$ similarly
and note that the covariance between the elements of $\hat{\bM}^{(1)}(r)$
and $\hat{\bM}^{(2)}(r)$ is also given by (\ref{eq:cov}). We note
that
\begin{equation}
D(r)\leq\E\left[\J(\hat{\bM}^{(1)}(r))\J(\hat{\bM}^{(2)}(r))\right],\label{eq:Dbd}
\end{equation}
by the following. Suppose that $\mathbf{A}$ is some random $\K\times N$
matrix and $X$ and $Y$ are two random variables such that $X$ is
independent of $(\mathbf{A},Y)$. Define $\mathbf{B}=\mathbf{A}+Xe_{\K,N}$
where $e_{\K,N}$ is the $\K\times N$ matrix with $1$ at the lower-right
corner and $0$ elsewhere. Recall that $\J(\mathbf{B})=\prod_{k=1}^{\K}\Theta_{k}(\mathbf{B})$
with $\Theta_{k}$ as defined below (\ref{eq:QM}). Hence, $\E\big[|\J(\mathbf{B})Y|\,\big|\,\mathbf{A},Y\big]=\E|Z\langle u,v+Xe_{N}\rangle|$,
for some random variable $Z$ and vectors $u$ and $v$ all measurable
w.r.t. $(\mathbf{A},Y)$ and $e_{N}=(0,\ldots,0,1)\in\R^{N}$. 

Fix some $r\in(-1,1)$. For any $t\in[-1,1]$ define the two jointly
Gaussian $\K\times N$ matrices $\mathbf{W}^{(1)}(t)$ and $\mathbf{W}^{(2)}(t)$
as follows. Elements not in the last column of both matrices (i.e.,
$j<N$) are correlated only if they have the same indices and in this
case, with $\rho_{N,k}(r)=\xi_{N,k}'(r)/r\xi_{N,k}'(1)$,
\[
\Big((\mathbf{W}^{(1)}(t))_{kj},(\mathbf{W}^{(2)}(t))_{kj}\Big)\sim\N\left(0,\left(\begin{array}{cc}
1 & t\rho_{N,k}(r)\\
t\rho_{N,k}(r) & 1
\end{array}\right)\right).
\]
For all $t$, the last column of $\mathbf{W}^{(1)}(t)$ and $\mathbf{W}^{(2)}(t)$
have the same joint distribution as the last column of $\hat{\bM}^{(1)}(r)$
and $\hat{\bM}^{(2)}(r)$, and are independent of all elements not
in the last column.

Define the function 
\[
\eta(t):=\E\left[\J(\mathbf{W}^{(1)}(t))\J(\mathbf{W}^{(2)}(t))\right]
\]
and note that $D(r)\leq\eta(r)$ since
\[
\eta(r)=\E\left[\J(\hat{\bM}^{(1)}(r))\J(\hat{\bM}^{(2)}(r))\right].
\]
The covariance matrix of $(\mathbf{W}^{(1)}(t),\mathbf{W}^{(2)}(t))$
satisfies the conditions of the following general lemma which we prove
in the appendix. Related less general inequalities are proved in \cite{BayatiMontanari2012,2nd}. 
\begin{lem}
\label{lem:GaussianBound}Suppose that $(X_{t},Y_{t})\sim\mathsf{N}(0,\Sigma(t))$
for any $t\in[-1,1]$, where for some $n\times n$ positive semi-definite
matrices $\Sigma_{0}$, $\Sigma_{1}$ and $\Sigma$, 
\begin{equation}
\Sigma(t):=\left(\begin{array}{cc}
\Sigma_{0} & \Sigma_{1}+t\Sigma\\
\Sigma_{1}+t\Sigma & \Sigma_{0}
\end{array}\right).\label{eq:XtYt}
\end{equation}
Let $f:\R^{n}\to\R$ be a (measurable) function such that $\E\left\{ f(X_{t})^{2}\right\} <\infty$.
Then there exist non-negative $\alpha_{0},\alpha_{1},\ldots$ such
that for any $t\in[-1,1]$,
\begin{equation}
\E\left\{ f(X_{t})f(Y_{t})\right\} =\sum_{k=0}^{\infty}\alpha_{k}t^{k}.\label{eq:Eff}
\end{equation}
\end{lem}

From the lemma $\eta(t)=\sum_{k=0}^{\infty}\alpha_{k}t^{k}$ for some
$\alpha_{k}\geq0$. Denote by $\hat{\mathbf{W}}^{(2)}(t)$ the matrix
obtained from $\mathbf{W}^{(2)}(t)$ by multiplying all but its last
column by $-1$. Note that $\J(\hat{\mathbf{W}}^{(2)}(t))=\J(\mathbf{W}^{(2)}(t))$
and that $(\mathbf{W}^{(1)}(t),\hat{\mathbf{W}}^{(2)}(t))$ and $(\mathbf{W}^{(1)}(-t),\mathbf{W}^{(2)}(-t))$
have the same law. Therefore, $\eta(t)=\eta(-t)$ and $\alpha_{k}=0$
for all odd $k$. Hence,
\begin{equation}
\eta(t)\leq\eta(0)+t^{2}(\eta(1)-\eta(0))\leq\eta(0)+t^{2}\eta(1).\label{eq:ineq1}
\end{equation}

Recall that $\J(\mathbf{W}^{(1)}(t))=\prod_{k=1}^{\K}\Theta_{k}(\mathbf{W}^{(1)}(t))$
with $\Theta_{k}$ as defined below (\ref{eq:QM}). Using Cauchy-Schwarz,
if we let $\chi_{j}$ be independent Chi variables with $j$ degrees
of freedom,
\begin{equation}
\eta(1)\leq\E\left[\J(\mathbf{W}^{(1)}(t))^{2}\right]=\prod_{k=0}^{\K-1}\E(\chi_{N-k}^{2}).\label{eq:ineq2}
\end{equation}

To bound $\eta(0)$ we will use a similar argument to the one used
in the proof of Lemma \ref{lem:general_bd}. Denote the last column
of $\mathbf{W}^{(i)}(0)$ by $Z^{(i)}$ and let $\mathbf{Q}^{(i)}$
be a $\K\times\K$ orthogonal matrix, measurable w.r.t. $Z^{(i)}$,
such that
\[
\mathbf{Q}^{(i)}\mathbf{W}^{(i)}(0)=\left(\begin{array}{cc}
\mathbf{A}^{(i)} & 0\\
V^{(i)} & \|Z^{(i)}\|
\end{array}\right),
\]
for some $\K-1\times N-1$ matrix $\mathbf{A}^{(i)}$ and row vector
$V^{(i)}$ of length $N-1$. Note that the elements of $\mathbf{A}^{(1)}$and
$\mathbf{A}^{(2)}$ and of $V^{(1)}$ and $V^{(2)}$ are all independent
of each other and of $\|Z^{(1)}\|$ and $\|Z^{(2)}\|$, and that they
all have the same distribution $\N(0,1)$. Since $\J(\mathbf{W}^{(i)}(0))=\J(\mathbf{Q}^{(i)}\mathbf{W}^{(i)}(0))$
and $\J(\mathbf{W}^{(i)}(0))=\prod_{k=1}^{\K}\Theta_{k}(\mathbf{W}^{(i)}(0))$,
in distribution 
\[
\J(\mathbf{W}_{N}^{(1)}(0))\J(\mathbf{W}_{N}^{(2)}(0))=\prod_{i=1,2}\Big(\sqrt{\|Z^{(i)}\|^{2}+\chi_{i,N-\K}^{2}}\prod_{k=1}^{\K-1}\chi_{i,N-k}\Big)
\]
where $\chi_{i,j}$ are independent Chi variables with $j$ which
are also independent of $\|Z^{(1)}\|$ and $\|Z^{(2)}\|$. Hence,
using Cauchy-Schwarz,
\begin{equation}
\eta(0)\leq\E(\chi_{N}^{2})\prod_{k=1}^{\K-1}(\E\chi_{N-k})^{2}.\label{eq:eta0}
\end{equation}

For $\bM(\bx)$ whose elements are i.i.d. $\N(0,1)$ variables, e.g.
by using that $\J(\bM(\bx))=\prod_{k=1}^{\K}\Theta_{k}(\bM(\bx))$,
\begin{equation}
(\E\J(\bM(\bx)))^{2}=\prod_{k=0}^{\K-1}(\E\chi_{N-k})^{2}.\label{eq:M}
\end{equation}

By combining (\ref{eq:ineq1}), (\ref{eq:ineq2}), (\ref{eq:eta0})
and (\ref{eq:M}), the left-hand side of (\ref{eq:Dbar-1}) is bounded
from above by
\begin{align*}
\frac{\eta(r)}{(\E\J(\bM(\bx)))^{2}} & \leq\frac{\E(\chi_{N}^{2})}{(\E\chi_{N})^{2}}+r^{2}\frac{\prod_{k=0}^{\K-1}\E(\chi_{N-k}^{2})}{\prod_{k=0}^{\K-1}(\E\chi_{N-k})^{2}}.
\end{align*}
The right-hand side is maximal for $\K=N$. Recall that $\E\chi_{j}=\sqrt{2}\,\frac{\Gamma((j+1)/2)}{\Gamma(j/2)}$
and $\E\chi_{j}^{2}=j$. By Gautschi's inequality, the first summand bounded by $N/(N-1)$. Using Stirling's approximation, one can show that the second summand is bounded by $r^2\sqrt{ N\pi /2}(1+o(1))$. Thus,  (\ref{eq:Dbar-1}) follows.
\end{proof}

\subsection*{Acknowledgements }

I am grateful to Mark Rudelson for an inspiring presentation of his
work with Barvinok \cite{RudelsonBarvinok} delivered in the Weizmann
Institute of Science, from which I have learned about the problem
studied in this paper. This project was supported by the Israel Science
Foundation (Grant Agreement No. 2055/21) and a research grant from
the Center for Scientific Excellence at the Weizmann Institute of
Science. The author is the incumbent of the Skirball Chair in New
Scientists.

\section*{Appendix: proof of Lemma \ref{lem:GaussianBound}}

Using that $\Sigma(1)$ is positive semi-definite, one can easily
verify that $\Sigma_{0}-\Sigma_{1}-\Sigma$ and therefore 
\[
\hat{\Sigma}(t):=\left(\begin{array}{cc}
\Sigma_{0}-\Sigma_{1} & t\Sigma\\
t\Sigma & \Sigma_{0}-\Sigma_{1}
\end{array}\right)
\]
are positive semi-definite. If $(\hat{X}_{t},\hat{Y}_{t})\sim\mathsf{N}(0,\hat{\Sigma}(t))$
and $Z\sim\mathsf{N}(0,\Sigma_{1})$, then $(\hat{X}_{t}+Z,\hat{Y}_{t}+Z)$
and $(X_{t},Y_{t})$ have the same law. By conditioning on $Z$ one
can check that if we can prove the lemma for $(\hat{X}_{t},\hat{Y_{t}})$
then it follows for the original pair $(X_{t},Y_{t})$. Hence, it
will be enough to prove the lemma assuming that $\Sigma_{1}=0$, as
we shall.

Next we will prove the lemma assuming in addition that $f$ is a polynomial
function, say of degree $d$. Denote the left-hand side of (\ref{eq:Eff})
by $\Lambda(t)$. By Wick's Theorem, $\Lambda(t)$ is also a polynomial
function. Hence, for any $t\in[-1,1]$,
\[
\Lambda(t)=\sum_{k=0}^{2d}\frac{1}{k!}\Lambda^{(k)}(0)t^{k},
\]
where $\Lambda^{(k)}(t)$ is the $k$-th derivative of $\Lambda(t)$.
We need to prove that $\Lambda^{(k)}(0)\geq0$. \WLOG~we may assume
that $\Sigma(t)$ is non-singular since by Wick's theorem $\Lambda^{(k)}(0)$
are continuous in the entries of $\Sigma_{0}$. For $k=0$, since
$X_{0}$ and $Y_{0}$ are i.i.d.,
\[
\Lambda^{(0)}(0)=\Lambda(0)=(\E f(X_{0}))^{2}\geq0.
\]

Suppose $\varphi_{C}(w)=\det(2\pi C)^{-1/2}\exp(-\frac{1}{2}w^{T}C^{-1}w)$
is the density of a centered Gaussian vector $W_{C}\in\R^{k}$ with
non-singular covariance matrix $C=(C_{ij})$. From integration by
parts and the well known fact that for $i\neq j$,\footnote{Where $\frac{\partial}{\partial C_{ij}}\varphi_{C}\left(x\right):=\lim_{\epsilon\to0}\frac{1}{\epsilon}\big(\varphi_{C+\epsilon(e_{ij}+e_{ji})}\left(x\right)-\varphi_{C}\left(x\right)\big)$,
with $e_{ij}$ being the standard basis to account for the symmetry
of $C$.} 
\[
\frac{\partial}{\partial C_{ij}}\varphi_{C}\left(x\right)=\frac{\partial}{\partial w_{i}}\frac{\partial}{\partial w_{j}}\varphi_{C}\left(w\right),
\]
one has that, for any continuously twice differentiable function $g:\,\mathbb{R}^{k}\to\mathbb{R}$
with polynomial growth rate as $|w|\to\infty$,
\[
\frac{\partial}{\partial C_{ij}}\mathbb{E}\left\{ g\left(W_{C}\right)\right\} =\int g\left(w\right)\frac{\partial}{\partial C_{ij}}\varphi_{C}\left(w\right)dw=\int\left(\frac{\partial}{\partial w_{i}}\frac{\partial}{\partial w_{j}}g\left(w\right)\right)\varphi_{C}\left(w\right)dw.
\]

Denote $f_{i_{1},\ldots,i_{k}}(x)=\frac{d}{dx_{1}}\cdots\frac{d}{dx_{k}}f(x)$.
By applying the above with $g(x,y)=f(x)f(y)$ and $C=\Sigma(t)$ we
obtain that 
\[
\frac{d}{dC_{i,n+j}}\E\left\{ f(X_{t})f(Y_{t})\right\} =\E\left\{ f_{i}(X_{t})f_{j}(Y_{t})\right\} ,
\]
and therefore
\begin{align*}
\Lambda'(t) & =\sum_{i,j=1}^{n}\Sigma_{ij}\E\left\{ f_{i}(X_{t})f_{j}(Y_{t})\right\} 
\end{align*}
and since $X_{0}$ and $Y_{0}$ are i.i.d., for $V=\left(\E f_{i}(X_{0})\right)_{i\leq n}$,
\begin{align*}
\Lambda'(0) & =\E V^{T}\Sigma V\geq0.
\end{align*}

More generally, by induction,
\begin{align*}
\Lambda^{(k)}(t) & =\sum_{i_{1},j_{1},\ldots,i_{k},j_{k}=1}^{n}\Sigma_{i_{1}j_{1}}\cdots\Sigma_{i_{k}j_{k}}\E\left\{ f_{i_{1},\ldots,i_{k}}(X_{t})f_{j_{1},\ldots,j_{k}}(Y_{t})\right\} ,
\end{align*}
and by independence
\begin{align*}
\Lambda^{(k)}(0) & =\E V_{k}^{T}(\Sigma\otimes\cdots\otimes\Sigma)V_{k}\geq0,
\end{align*}
where $\otimes$ is the Kronecker product and $V_{k}$ is the vectorization
of $(\E f_{i_{1},\ldots,i_{k}}(X_{0}))_{i_{1},\ldots,i_{k}=1}^{n}$
in the appropriate order to match the summation above, and the inequality
follows since the Kronecker product of positive semi-definite matrices
is positive semi-definite. This completes the proof for polynomial
$f$.

Next, assume that $(X_{t},Y_{t})$ is as in the lemma and $f$ is
an arbitrary function such that $\E\left\{ f(X_{t})f(Y_{t})\right\} \leq\E\left\{ f(X_{t})^{2}\right\} <\infty$.
Let $f_{i}$ be a sequence of polynomial functions such that $\|f(X_{t})-f_{i}(X_{t})\|_{2}\overset{i\to\infty}{\longrightarrow}0$,
where $\|Y\|_{p}=(\E|Y|^{p})^{1/p}$. The existence of such sequence
can be easily proved using the fact that Hermite polynomials are dense
in the $L^{2}$ space of the Gaussian measure. For some non-negative
$\alpha_{0}^{(i)},\alpha_{1}^{(i)},\ldots$ and $t\in[-1,1]$, $\E\left\{ f_{i}(X_{t})f_{i}(Y_{t})\right\} =\sum_{k=0}^{\infty}\alpha_{k}^{(i)}t^{k}$.
Note that
\begin{align*}
 & \Big|\E\left\{ f(X_{t})f(Y_{t})\right\} -\E\left\{ f_{i}(X_{t})f_{i}(Y_{t})\right\} \Big|\\
 & \leq\|f(X_{t})f(Y_{t})-f_{i}(X_{t})f(Y_{t})\|_{1}+\|f_{i}(X_{t})f(Y_{t})-f_{i}(X_{t})f_{i}(Y_{t})\|_{1}\\
 & \leq\|f(X_{t})-f_{i}(X_{t})\|_{2}\|f(Y_{t})\|_{2}+\|f(Y_{t})-f_{i}(Y_{t})\|_{2}\|f_{i}(X_{t})\|_{2}.
\end{align*}
Hence, 
\[
\sup_{t\in[-1,1]}\Big|\E\left\{ f(X_{t})f(Y_{t})\right\} -\sum_{k=0}^{\infty}\alpha_{k}^{(i)}t^{k}\Big|\overset{i\to\infty}{\longrightarrow}0.
\]

Using the fact that $\alpha_{k}^{(i)}\geq0$ for all $k$ and $i$,
it is a standard exercise to prove that $\alpha_{k}^{(i)}\overset{i\to\infty}{\longrightarrow}\alpha_{k}\geq0$
for some $\alpha_{k}$ and $\E\left\{ f(X_{t})f(Y_{t})\right\} =\sum_{k=0}^{\infty}\alpha_{k}t^{k}$.
This completes the proof.\qed

\bibliographystyle{plain}
\bibliography{master}

\end{document}